\documentclass[10pt,a4paper]{amsart}
\usepackage[utf8]{inputenc}
\usepackage[T1]{fontenc}
\usepackage{amsmath}
\usepackage{upquote}
\usepackage{bm}
\usepackage{amsfonts}
\usepackage{amssymb}
\usepackage[english]{babel}
\usepackage{tabularx}
\usepackage{mathtools}
\usepackage{adjustbox}
\usepackage{ytableau}
\usepackage{graphicx}
\usepackage{hyperref}
\usepackage{xcolor}
\usepackage{diagbox}
\usepackage{tikz}
\usepackage{chngcntr}
\usepackage{array}
\usepackage{comment}
\usepackage{booktabs}

\newcolumntype{H}{>{\setbox0=\hbox\bgroup}c<{\egroup}@{}}
\usetikzlibrary{calc}
\numberwithin{equation}{section}
\newtheorem{thm}{Theorem}[section]
\newtheorem{pr}[thm]{Proposition}
\newtheorem{lm}[thm]{Lemma}
\newtheorem{re}[thm]{Remark}
\newtheorem{df}[thm]{Definition}
\newtheorem{ex}[thm]{Example}
\newtheorem{cor}[thm]{Corollary}

\newtheorem{con}[thm]{Conjecture}
\newtheoremstyle{case}{}{}{}{}{}{:}{ }{}
\theoremstyle{case}
\newtheorem{case}{Case}
\counterwithin*{case}{thm}
\newtheoremstyle{caso}{}{}{}{}{}{:}{ }{}
\theoremstyle{caso}

\newcommand{\lcm}{\text{lcm}}

\newcommand{\Z}{\mathbb{Z}}
\newcommand{\Q}{\mathbb{Q}}
\newcommand{\R}{\mathbb{R}}

\newcommand{\floor}[1]{\left\lfloor #1 \right\rfloor}

\title{Rectangle partitions generalizing integer partitions}
\author{Krystian Gajdzica, Robin Visser, and Maciej Zakarczemny}

\address{Theoretical Computer Science Department \\ Faculty of Mathematics and Computer Science\\ Jagiellonian University\\ Łojasiewicza 6\\ 30-348 Kraków\\ Poland}
\email{krystian.gajdzica@uj.edu.pl}

\address{Department of Algebra\\ Faculty of Mathematics and Physics\\ Charles University \\ Sokolovsk\'a 83\\ 186 75 Praha 8\\ Czech Republic}
\email{robin.visser@matfyz.cuni.cz}

\address{Department of Applied Mathematics\\ Cracow University of Technology\\ Warszawska 24\\
31-155 Kraków\\ Poland}
\email{maciej.zakarczemny@pk.edu.pl}



\begin{document}

\setlength{\parindent}{10mm}

\begin{abstract}
In this paper, we introduce a natural geometric extension of the partition function. More precisely, we investigate the problem of counting partitions of a rectangle into rectangular blocks with integer sides. Here, two partitions of a rectangle are indistinguishable if they consist of the same multiset of blocks, their geometric arrangement does not matter.
\end{abstract}
\maketitle

\section{Introduction}
We begin by presenting two distinct rectangle tiling problems.
Consider the function \( p(m, n) \), which represents the number of ways to partition a rectangle of size \( m \times n \) into a set of rectangular blocks with integer sides, assuming that two partitions are indistinguishable if their multisets of blocks are identical. Similarly, one can also count the number of compositions of a rectangle of size $m \times n$ into a set of rectangular blocks with integer sides. Here, in contrast to rectangle partitions, the geometric arrangement of blocks is taken into account. The relation between the two definitions is analogous to the relation between integer partitions and integer compositions. In the literature, special cases of rectangle compositions are discussed, such as tilings with dominoes (see, for instance, \cite[Section 7.7]{GKP}). In our work, we focus on partitions.

Before we introduce the exact definition of a partition of a rectangle, let us fix some notions. Throughout the paper, the symbols $\mathbb{N}_0$, $\mathbb{N}$, $\mathbb{Q}$ and $\mathbb{R}$ refer to the set of non-negative integers, the set of  positive integers, the set of rational numbers and the set of real numbers, respectively.

\begin{df}
A rectangle in $\mathbb{R}^2$ is a set of the form $
[a,b]\times[c,d], a<b, c<~d,$\\
$ a,b,c,d \in \mathbb{R},
$
whose sides are parallel to the coordinate axes. \end{df}

\begin{df}
Let $R = [0, m] \times [0, n] \subset \mathbb{R}^2$, where $m, n \in \mathbb{N}$. A partition of the rectangle $R$ is defined as a multiset of rectangles
$\mathcal{P} = \{ R_i \}_{i \in I}$ such that:
\begin{enumerate}
    \item each rectangle $R_i$ is with positive integer side lengths;
    \item a rectangle of size $k\times l$ is indistinguishable  with a rectangle of size $l\times k$;  
    \item all the rectangles from $\mathcal{P}$ can be arranged in such a way that they cover exactly the area of the rectangle $R$ and their interiors do not intersect. 
\end{enumerate}
 Moreover, two partitions are considered to be the same if they are equal as multisets.\\
 Denote by $ p(m, n) $ the number of partitions of a rectangle of size $ m \times n $.
\end{df}

\begin{re}\rm{ Note that $p(n, m) = p(m, n), $ where $m, n \in \mathbb{N}$.} Moreover, the function $p$ is monotone increasing in each coordinate.
\end{re}
\begin{re}\rm{
    The number of partitions \( p(1, n) \) correspond to partitions of a rectangle of size \( 1 \times n \) into rectangles of size \( 1 \times k \) for any $k\in\{1,2,\ldots,n\}$. In other words, $p(1,n)=p(n)$, where $p(n)$ refers to the well-known Euler partition function.}
\end{re}

\begin{ex} All $p(2,3) = 10$ possible partitions of the $2 \times 3$ rectangle present as follows:
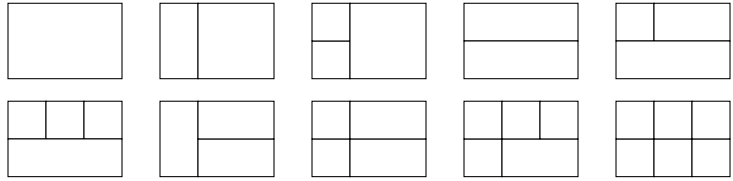
\begin{align*}
    \begin{tikzpicture}[scale=0.5]  
    \draw[color=black, fill = white] (4,4.5) rectangle (7,6.5);
    \draw[color=black, fill = white] (9,4.5) rectangle (8,6.5);
    \draw[color=black, fill = white] (11,4.5) rectangle (9,6.5);
    \draw[color=black, fill = white] (12,4.5) rectangle (13,5.5);
    \draw[color=black, fill = white] (12,5.5) rectangle (13,6.5);
    \draw[color=black, fill = white] (15,4.5) rectangle (13,6.5);
    \draw[color=black, fill = white] (16,5.5) rectangle (19,6.5);
    \draw[color=black, fill = white] (16,4.5) rectangle (19,5.5);
    \draw[color=black, fill = white] (20,5.5) rectangle (21,6.5);
    \draw[color=black, fill = white] (21,5.5) rectangle (23,6.5);
    \draw[color=black, fill = white] (20,4.5) rectangle (23,5.5);
  \end{tikzpicture}\\
  \begin{tikzpicture}[scale=0.5] 
      \draw[color=black, fill = white] (4,5.5) rectangle (5,6.5);
    \draw[color=black, fill = white] (5,5.5) rectangle (6,6.5);
    \draw[color=black, fill = white] (6,5.5) rectangle (7,6.5);
    \draw[color=black, fill = white] (4,4.5) rectangle (7,5.5);
    \draw[color=black, fill = white] (8,4.5) rectangle (9,6.5);
    \draw[color=black, fill = white] (11,4.5) rectangle (9,5.5);
    \draw[color=black, fill = white] (11,5.5) rectangle (9,6.5);
    \draw[color=black, fill = white] (13,4.5) rectangle (12,5.5);
    \draw[color=black, fill = white] (13,5.5) rectangle (12,6.5);
    \draw[color=black, fill = white] (15,4.5) rectangle (13,5.5);
    \draw[color=black, fill = white] (15,5.5) rectangle (13,6.5);
    \draw[color=black, fill = white] (16,4.5) rectangle (17,5.5);
    \draw[color=black, fill = white] (16,5.5) rectangle (17,6.5);
    \draw[color=black, fill = white] (19,4.5) rectangle (17,5.5);
    \draw[color=black, fill = white] (17,5.5) rectangle (18,6.5);
    \draw[color=black, fill = white] (18,5.5) rectangle (19,6.5);
    \draw[color=black, fill = white] (20,4.5) rectangle (21,5.5);
    \draw[color=black, fill = white] (20,5.5) rectangle (21,6.5);
    \draw[color=black, fill = white] (22,4.5) rectangle (21,5.5);
    \draw[color=black, fill = white] (22,5.5) rectangle (21,6.5);
    \draw[color=black, fill = white] (23,4.5) rectangle (22,5.5);
    \draw[color=black, fill = white] (23,5.5) rectangle (22,6.5);
  \end{tikzpicture}
\end{align*}
\end{ex}

One can consider a special case corresponding to ``square partitions'' $p(n, n)$. Currently, the values of $p(n, n)$ are known up to $n = 7$, namely: (1, 4, 21, 192, 2035, 27407, 399618), see \cite{OEIS1}. The initial values of $p(2,n)$, on the other hand, are known up to $n=40$, and present as follows $(1, 2, 4, 10, 22, 44, 91, 172, 326, 595, \ldots, 145550924)$, see \cite{OEIS2}. To view specific tilings, one can use the application \cite{Zak1}.

There are a few main results of this paper concerning partitions of the rectangle $2\times n$. The first of them gives us the exact information about the asymptotic of $p(2,n)$ in the spirit of Hardy and Ramanujan.

\begin{thm}\label{thm: asymptotic p(2,n)}
    If $n \to \infty$, then
    $$p(2, n) \sim \frac{\pi \sqrt[4]{2}}{32 n^{7/4}}   \exp \left(\pi \sqrt{2n} \right).$$
\end{thm}
As an application of our asymptotic formula for the rectangular partition function, we establish that $p(2,n)$ obeys Benford’s Law \cite{ARS,DO,Luca}.\\
We also derive an asymptotic formula for the number of symmetric rectangular partitions, obtained by considering horizontally symmetric sets of blocks.\\
Similarly to the case of classical partitions, one can also investigate partitions of a rectangle under some additional restrictions. In particular, we may consider partitions of the rectangle $2\times n$ into blocks of sizes $1\times1, 1\times2,\ldots,1\times k$ and $2\times2,2\times3,\ldots,2\times l$, where $k$ and $l$ are fixed positive integers. Throughout the paper, we denote their number by $p_{k,l}(2,n)$. At this point, it should be pointed out that $p_{k,1}(2,n)$ counts the number of partitions of the rectangle $2\times n$ into blocks of sizes $1\times1,1\times2,\ldots,1\times k$. It turns out that $p_{k,l}(2,n)$ is of polynomial growth. 

\begin{thm}\label{thm: asymptotics p_k,l}
    For $k,l\geq1$, we have that
    \begin{align*}
        p_{k,l}(2,n)=\frac{2^{k-1}n^{k+l-2}}{l!k!(k+l-2)!}+O(n^{k+l-3}).
    \end{align*}
\end{thm}


Another interesting example of restricted partitions are so-called $m$-ary partitions. Let us recall that for a given integer $m\geq2$, an $m$-ary partition of $n$ is any partition into parts from the set $\{m^i:i\in\mathbb{N}_0\}$. Their total number for fixed $n$ is denoted by $b_m(n)$. In this paper, we consider the rectangular generalization of the function $b_m(n)$. More precisely, we study the number $b_{i,j}(2,n)$ of the partitions of the rectangle $2\times n$, where only blocks of sizes $1\times1,1\times m,\ldots,1\times m^i$ and $2\times m,2\times m^2,\ldots,2\times m^j$ are permitted. However, instead of asymptotic of $b_{i,j}(2,n)$, our goal is to derive the analogue of Alkauskas' result (see Theorem \ref{thm: generalization of Alkauskas}).

 Let us describe the content of the paper in some details. Section $2$ presents some basic estimates for $p(2,n)$ in terms of the partition function $p(n)$. In Section $3$, we derive the exact asymptotic formula for $p(2,n)$ and prove Theorem \ref{thm: asymptotic p(2,n)}. Furthermore, we apply Theorem \ref{thm: asymptotic p(2,n)} in practice to show that the sequence $p(2,n)$ obeys the so-called Benford's Law. Next, the asymptotic of the number of horizontally symmetric partitions of the rectangle $2\times n$ is investigated in Section $4$. Section $5$ is devoted to restricted partitions of a rectangle and the asymptotics of $p_{k,l}(2,n)$. Finally, Section $6$  deals with the rectangular generalization of the $m$-ary partition function, and prove an analogue of Alkauskas' result concerning the behavior of $b_m(n)$ modulo $m$ (see Theorem \ref{thm: Alkauskas}).

\section{The lower and the upper bounds for \texorpdfstring{$p(2,n)$}{p(2,n)} }

In this section, we present an elementary combinatorial reasoning showing that the value of $p(2,n)$ can be bounded from above and below in terms of the partition function. We have the following.

\begin{thm}\label{t1}
For $n \geq 1$, the number $p(2,n)$ of tilings of the $2 \times n$ rectangle by integer-sided rectangles satisfies:
\begin{align*}
    p(n)+\sum_{i=1}^{n}p(i) \leq p(2,n) \leq \sum_{i=0}^n \binom{p(n-i)+1}{2}p(i).
\end{align*}
\end{thm}

\begin{proof}
At first, let us deal with the lower bound. There is a certain relationship between $p(2, n)$ and $p(n)$. For the function $p(2,n)$, we can formulate the inequality:
$$
p(2, n) \geq p(2, n-1) + p(n), \quad n \geq 1.
$$
On the right-hand side, we count two types of partitions: those with the $1 \times n$ block and those without the $1 \times n$ block, but instead with a $1 \times 2$ block placed in the vertical position. The number $p(2, n-1)$ is the number of partitions of the rectangle $2 \times (n-1)$, where the vertical block $2 \times 1$ is added to obtain the $2 \times n$ rectangle. On the other hand, $p(n)$ is the number of partitions of the rectangle $1 \times n$, which corresponds to attaching a horizontal block $1 \times n$ to either the lower or upper half of the $2 \times n$ rectangle, leaving the other half to be partitioned as $1 \times n$ (a one-dimensional partition).

Now, if we subtract the partitions corresponding to the terms $p(2, n-1)$ and $p(n)$ from $p(2, n)$, the remaining partitions of the rectangle $2 \times n$ are those that cannot be obtained by either adding a vertical block $2 \times 1$ or by inserting a horizontal block $1 \times n$. The set of remaining partitions includes those into rectangles of the form $2 \times k$, where $k > 1$. These correspond bijectively to the partitions of $1 \times n$ into rectangles of the form $1 \times k$, where $k > 1$, that is, to the usual numerical partitions $q(n)$ into parts greater than 1. In other words, we have that
\begin{align*}
    \sum_{n=0}^\infty q(n)x^n &= \prod_{k=2}^\infty \frac{1}{1 - x^k} = (1 - x)\prod_{k=1}^\infty \frac{1}{1 - x^k}\\
 &=(1 - x)\sum_{n=0}^\infty p(n)x^n = p(0) + \sum_{n=1}^\infty (p(n) - p(n-1))x^n.
\end{align*}

Hence, $q(n) = p(n) - p(n-1)$ for $n \geq 1$, and we obtain that
$$
p(2, n) \geq p(2, n-1) + 2p(n) - p(n-1), \quad n \geq 1.
$$
Summing this recurrence inequality from $1$ to $n$, and noting that $p(2,1) = 2p(1)$ leads us to the conclusion that
$$
p(2, n) \geq p(n) + \sum_{i=1}^{n} p(i),
$$
as required.

For the upper bound, due to the horizontal symmetry of the $2 \times n$ rectangle, we can get the following inequality
$$
p(2, n) \leq \sum_{i=0}^n p(i) \binom{p(n-i)+1}{2}.
$$
First, we fill the $2 \times i$ area of the rectangle with blocks of type $2 \times k$ for $k\in\{1,2,\ldots,i\}$ there are $p(i)$ different ways to do that. The remaining area of the rectangle is filled with blocks of type $1 \times k$. We require that a pair of partitions divide the remaining part of the rectangle into the upper and the lower halves. It does not matter which partition is placed at the top and which at the bottom. Hence, we get at most $\binom{p(n-i)+1}{2}$ distinct ways to do that.
Summing over all $i \in \{0, \ldots, n\}$, we obtain the claimed upper bound.
\end{proof}

As a direct consequence of the proof of the lower bound for $p(2,n)$, we get the following.
\begin{cor}
For every $n\geq 1$, we have that
$$p(2, n) \geq p(2, n-1) + 2p(n) - p(n-1).$$
\end{cor}

At this point, it is worth pointing out that there are various ways to derive the bounds for the value of $p(2,n)$. For example, one can present a modification of the reasoning presented in the proof of Theorem \ref{t1} to deduce the following estimates.

\begin{pr}\label{t2}
For \( n \geq 1\), we have that:
$$
p(2, n) \leq p(n)^2 + \sum_{i=1}^n \left( p(i) + p(i-1) \right) p(n-i)^2.
$$
\end{pr}

\begin{proof}
At first, let us notice that one can always take two blocks $1\times 2$ placed vertically and rotate them to put them in horizontal position. Hence, we can assume that the number of \( 1\times 2\) blocks placed vertically is at most $1$.
It follows that
$$
p(2, n) \leq \sum_{i=0}^n p(i)p(n-i)^2 + \sum_{i=1}^n p(i-1)p(n-i)^2,
$$
where the right-hand side corresponds to constructing blocks of sizes \( 2 \times k \) for $k\in\{2,3,\ldots,i\}$, grouped up to \( i \). 
 The first term on the right hand side accounts for the cases where the number of \( 1 \times 2 \) blocks is even, and all such blocks are arranged horizontally; here \( 0 \leq i \leq n \). 
The second term, on the other hand, corresponds to the cases where the number of \( 1 \times 2 \) blocks is odd, with one block placed vertically and the remaining blocks arranged in horizontal positions; here \( 1 \leq i \leq n \). Hence, we obtain that
$$
p(2, n) \leq p(n)^2 + \sum_{i=1}^n \left( p(i) + p(i-1) \right) p(n-i)^2.
$$
\end{proof}

A few questions arise from the above discussion. For instance, one can ask whether there is a possibility to generalize Theorem \ref{t1} to any function $p(m,n)$ with the fixed parameter $m\geq3$. Further, we see that the estimations from neither Theorem \ref{t1} nor Proposition \ref{t2} are tight. Therefore, it is natural to try to find out, by utilizing various methods, more precise bounds for $p(2,n)$.

\section{Asymptotics for \texorpdfstring{$p(2,n)$}{p(2,n)}}

In this section, we prove one of the main results of this paper, namely, Theorem \ref{thm: asymptotic p(2,n)}.

\begin{proof}[Proof of Theorem \ref{thm: asymptotic p(2,n)}]
By a similar argument to those from Section $2$, any partition of a $2 \times n$ rectangle can be rearranged such that all blocks of type $2 \times k$ for $k \geq 2$ are shifted to the left side of the rectangle. This implies that, for any such partition of $2 \times n$, it can be rearranged such that there exists a unique $i = 0, \dots, n$, where the first $i$ columns of the rectangle gives a partition of $2 \times i$ using only blocks of type $2 \times k$ for $k \in \{2, 3, \dots, i\}$, and the remaining $n-i$ columns of the rectangle gives a partition of $2 \times (n-i)$  using only blocks of type $1 \times k$.

To count the possible number of partitions of the left $2 \times i$ rectangle, this is simply the number of non-unitary partitions of $i$, i.e. the number of partitions of $i$ without using 1 (sometimes also called \emph{nuclear} partitions).  This is easily seen to be $p(i) - p(i-1)$.  Now, let $\widetilde{p}(2,n)$ be the number of partitions of $2 \times n$ using only blocks of the form $1 \times k$. By the above argument, we therefore have
\begin{equation*}
        p(2,n) = \sum_{i=0}^n (p(i) - p(i-1)) \widetilde{p}(2, n-i) .
\end{equation*}

The asymptotic for the number of non-unitary partitions, $p(n) - p(n-1)$, is well-known, e.g. see \cite[Theorem~1]{AGHHPSS}. In particular, we have that
    \begin{equation*}
        p(n) - p(n-1) \sim \frac{\pi}{12 \sqrt{2} n^{3/2}} \exp{\left( \pi \sqrt{\frac{2n}{3}}  \right) }
    \end{equation*}
as $n \to \infty$.  To derive an asymptotic for $\widetilde{p}(2, n)$, we shall use a result of  Erd\H{o}s--Nicolas--S\'{a}rk\"{o}zy \cite{ENS} on the number of partitions of $n$ with no subsum equalling $a$.  In particular, let $R(2n, n)$ be the number of partitions $(\lambda_1, \lambda_2, \dots, \lambda_k)$ of $2n$ with the property that no subpartition ($\lambda_{i_1}, \lambda_{i_2}, \dots, \lambda_{i_j})$ has its sum $\lambda_{i_1} + \dots + \lambda_{i_j}$ equalling $n$. Note that we have the bounds
    \begin{equation} \label{eq:p2n_asymp_bounds}
        p(2n) - R(2n, n) \leq  \widetilde{p}(2, n) \leq p(2n) ,
    \end{equation}
 as clearly any partition of $2 \times n$ into blocks of the form $1 \times k$ yields a (standard) partition of $2n$.  Conversely, any (standard) partition of $2n$ that can be obtained by adding together two (not necessarily distinct) partitions of $n$, yields a valid rectangular partition of $2 \times n$ using blocks of the form $1 \times k$.

Erd\H{o}s--Nicolas--S\'{a}rk\"{o}zy \cite[p.~206]{ENS} proved that $R(2n, n) = p(n)^{1 + o(1)}$, which in particular implies that $R(2n, n) = o(p(2n))$ as $n \to \infty$. Thus, the bounds in (\ref{eq:p2n_asymp_bounds}) give us the following asymptotic for $\widetilde{p}(2, n)$:
     \begin{equation*}
        \widetilde{p}(2, n) \sim p(2n) \sim  \frac{1}{8n \sqrt{3}}  \exp{\left(  2 \pi \sqrt{\frac{n}{3}} \right)} ,
    \end{equation*}
where we have used the Hardy--Ramanujan asymptotic for $p(n)$.
    
Finally, in order to derive an asymptotic for $p(2, n)$, we shall use a theorem of Murty \cite[Theorem~3]{Murty}, which utilizes Laplace's saddle point method to give an asymptotic for the convolution of two power series.
    In particular, we have that
    \begin{equation*}
        p(2, n) \sim c_f c_g 2 \sqrt{2\pi} \frac{A^{2 \alpha + 1} B^{2 \beta+1}}{(A^2 + B^2)^{\alpha + \beta + 5/4}} n^{\alpha + \beta + 3/4} \exp{\left( \sqrt{(A^2 + B^2) n }  \right)}
    \end{equation*}
where $\alpha, \beta, A, B, c_f, c_g \in \R$ are constants such that
    \begin{equation*}
        p(n) - p(n-1) \sim c_f n^\alpha \exp \left( A \sqrt{n} \right), \quad \text{and} \quad \widetilde{p}(2, n) \sim c_g n^\beta \exp{\left(   B \sqrt{n} \right)}.
    \end{equation*}

Substituting our computed values of $\alpha = -3/2 $, $\beta = -1$, $A = \pi \sqrt{2/3} $, $B = 2 \pi /\sqrt{3}$, $c_f = \pi / (12 \sqrt{2})$, and $c_g = 1/(8 \sqrt{3})$, we obtain the following asymptotic:

    \begin{align*}
        p(2, n) &\sim \frac{\pi}{96 \sqrt{6}} 2 \sqrt{2 \pi} \frac{(\pi \sqrt{2/3})^{-2} (2 \pi / \sqrt{3})^{-1}}{( \pi ^2)^{-5/4}} n^{-7/4} \exp{\left( \sqrt{ 2 \pi ^2 n }  \right)} \\
        &\sim  \frac{\pi \sqrt[4]{2}}{32 n^{7/4}}   \exp \left(\pi \sqrt{2n} \right),
    \end{align*}
    as claimed.
\end{proof}

At this point, Theorem \ref{thm: asymptotic p(2,n)} motivates us to state a conjecture regarding the asymptotic of $p(3,m)$.

\begin{con}
There exist constants $k_3,\lambda_3>0$, $\delta_3\in\mathbb{R}$ such that
$$
p(3,n)\sim k_3n^{\delta_3}\exp\bigl(\lambda_3\sqrt{n}\bigr)\text{\,\,as }n\to\infty.
$$
\end{con}

\subsection{The rectangular partition function \texorpdfstring{$p(2,n)$}{p(2,n)} abides by Benford's Law}

We were inspired to write this part by Florian Luca’s presentation at the 33èmes Journées Arithmétiques conference.
We show that the rectangular partition function $p(2,n)$ obeys the so-called Benford's Law. More precisely, let us fix a sequence $\left(a_n\right)_{n\geq1}$ of integers and define 
\begin{align*}
    B_a(f,b;x):=\frac{|\{1\leq n\leq  x: a_n\text{ in base } b \text{ begins with the string }f\}|}{x}.
\end{align*}
In such a setting, we say that $(a_n)_{n\geq1}$ satisfies \emph{Benford's Law} if for every string of digits $f$ in base $b$ we have

$$
\lim_{x\to\infty}B_a(f,b;x)=\log_b(f+1)-\log_b(f) \pmod{1}.
$$
Here, we interpret $f$ as an integer in base $b$. That is, if $f=f_0f_1\cdots f_{t-1}$ has $t$ digits in $\{0,1,\dots,b-1\}$ with $f_0\ne 0$, then
$$
f=f_0 b^{t-1}+f_1 b^{t-2}+\cdots+f_{t-1}.
$$

In other words, the sequence $(a_n)$ satisfies Benford’s Law if the sequence of fractional parts $\{\log_b(a_n)\}$ is uniformly distributed mod $1$, in the sense that for every interval $[a,b)\subset[0,1)$,  
$$
\lim_{N\to\infty}\frac{1}{N}\,\#\{\,n\leq N : \{\log_b(a_n)\}\in[a,b)\,\}=b-a.
$$  
For a given digit string $f$ in base $b$, the condition that an integer begins with $f$ is equivalent to requiring that $\{\log_b(a_n)\}\in [\log_b(f),\log_b(f+1))$. The length of this interval is exactly $\log_b(f+1)-\log_b(f)$, which gives the Benford probability.

In 2011, Anderson, Rolen and Stoehr \cite{ARS} showed that the partition function $p(n)$ obeys Benford's Law in every base $b$. More recently, Douglass and Ono \cite{DO} (see, also \cite{Luca}) investigated the plane partition function in that direction. In particular, Anderson et al. \cite{ARS} discovered an efficient criterion for a sequence of positive integers to satisfy  Benford's Law. In order to perform their result, we need an additional definition.

\begin{df}\label{def: Good sequence}
An integer sequence $(a_n)_{n\ge 1}$ is said to be \emph{good} if
$$
a(n) \sim b(n) e^{c(n)} \quad \text{as} \quad n \to \infty,
$$
and the following conditions are satisfied:
\begin{enumerate}
\item There exists some integer $h \geq 1$ such that $c(n)$ is $h$-differentiable and $c^{(h)}(n) \to 0$ monotonically for all large values of $n$.
\item We have that
$$
\lim_{n\to+\infty} n \lvert c^{(h)}(n) \rvert = +\infty.
$$
\item We have that
$$
\lim_{n\to+\infty} \frac{D^{(h)} \log b(n)}{c^{(h)}(n)} = 0,
$$
where $D^{(h)}$ denotes the $h$ derivative.
\end{enumerate}
\end{df}

Now, it turns out that the following criterion holds.

\begin{thm}\label{thm: Anderson}
Good integer sequences abide by Benford's Law in any integer base $b \geq 2$.
\end{thm}

As a consequence of Theorem \ref{thm: Anderson}, we obtain the following property of the rectangular partition function $p(2,n)$.

\begin{thm}
The function $p(2,n)$ abides by Benford's Law in any integer base $b \geq 2$.
\end{thm}

\begin{proof}
    Theorem \ref{thm: Anderson} asserts that it is enough to verify conditions from Defnition \ref{def: Good sequence} one-by-one. In order to do that, let us apply Theorem \ref{thm: asymptotic p(2,n)}, which states that
    \begin{align*}
        p(2, n) \sim \frac{\pi \sqrt[4]{2}}{32 n^{7/4}}   \exp \left(\pi \sqrt{2n} \right).
    \end{align*}
    Hence, we set
    $$
b(n) := \frac{\pi\sqrt[4]{2}}{32\,n^{7/4}}, \qquad c(n) := \pi\sqrt{2n}.
$$
Now, if we put $h:=1$, then we get that
$$
c'(n) = \frac{\pi}{\sqrt{2n}} \to 0 \ \text{monotonically},
$$
as the condition $(1)$ requires. Further, it follows that
$$
n\,c'(n) = \pi\sqrt{\frac{n}{2}} \to +\infty,
$$
which implies that the condition $(2)$ is satisfied. Finally, we obtain that
\begin{align*}
\log b(n) = \log\left(\frac{\pi\sqrt[4]{2}}{32}\right) - \frac{7}{4} \log n    
\end{align*}
and
$$
\frac{D'(\log b(n))}{c'(n)} =
\frac{-\frac{7}{4}\cdot\frac{1}{n}}{\frac{\pi}{\sqrt{2n}}}
= -\frac{7\sqrt{2}}{4\pi} \cdot \frac{1}{\sqrt{n}} \to 0,
$$
as $n\to\infty$, so the condition $(3)$ holds. This concludes the proof.
\end{proof}

\section{Horizontally symmetric set of blocks in a \texorpdfstring{$2 \times n$}{2 x n} rectangle}
Here, we consider the number of distinct sets of $2 \times k$ and $1 \times k$ blocks that fill a $2 \times n$ rectangle in such a way that a horizontally symmetric arrangement is possible. At first, we examine the configurations that can be transformed into a horizontally symmetric arrangement and do not contain any $1 \times 2$ blocks. Therefore, let  us denote by $S(n)$ the number of distinct multisets (up to the types of blocks used, disregarding the actual placement) that tile the $2 \times n$ rectangle. In fact, we deal with the sequence $S(n)$, which satisfies the following recurrence equation:
$$
n S(n) = - \sum_{k=1}^{n} S(n - k) -\sum_{k=1}^{\lfloor n/2 \rfloor} 2 S(n - 2k) +2 \sum_{v=1}^{n} \sum_{k=1}^{\lfloor n/v \rfloor} v S(n - kv),
$$
with the initial conditions $S(0) = 1$, $S(1) = 1$, $S(2) = 2$, $S(3) = 4$. This sequence corresponds to the number of partitions of $n$ in which each number $k \geq 3$ can appear in two distinct types, while for $k \in\{ 1, 2\}$ only one type is allowed. The generating function for $S(n)$ has the form:
$$
G(x) = \prod_{k=1}^{2} \frac{1}{1 - x^k} \cdot \prod_{k=3}^{\infty} \left( \frac{1}{1 - x^k} \right)^2 = \frac{1}{(1 - x)(1 - x^2)} \prod_{k=3}^{\infty} \frac{1}{(1 - x^k)^2}.
$$
In other words, we have that
$$
G(x) = P(x)^2(1 - x)(1 - x^2), \quad \text{where} \quad P(x) = \prod_{k=1}^{\infty} \frac{1}{1 - x^k}.
$$
This implies the following formula for $S(n)$:

\begin{align*}
S(n) = \sum_{k=0}^{n} p(k)p(n - k) &- \sum_{k=0}^{n-1} p(k)p(n - 1 - k)\\ 
&- \sum_{k=0}^{n-2} p(k)p(n - 2 - k) + \sum_{k=0}^{n-3} p(k)p(n - 3 - k).    
\end{align*}

Now, let us include all the blocks \( 1 \times 2 \), and denote by $T(n)$ the number of distinct sets of \( 2 \times k \) and \( 1 \times k \) blocks that fill a \( 2 \times n \) rectangle such that a horizontally symmetric arrangement of blocks is possible. It is not difficult to see that the generating function of the sequence $T(n)$ corresponds to the generating function of the partial sums \( S(0) + \dots + S(n) \), i.e.,
\begin{equation}\label{GTP2}
G_T(x) = P(x)^2 (1 - x^2)
\end{equation}

Hence, we obtain the recurrence relation of the form
$$
T(n) = \sum_{k=0}^{n} p(k)p(n - k) - \sum_{k=0}^{n - 2} p(k)p(n - 2 - k),
$$
with the initial conditions:
$$
T(0) = 1,\quad T(1) = 2.
$$
From this, one can derive the asymptotic of the type
$T(n) \;\sim\; \frac{C}{n^r} \exp\left(\alpha \sqrt{n} \right).$ We have the following.

\begin{thm}\label{ts1}
If $n\to \infty,$ then
$$
T(n)\sim\frac{\pi}{6\sqrt[4]{3}}n^{-7/4}\;\exp\!\Bigl(2\pi\sqrt{\tfrac{n}{3}}\Bigr).
$$
\end{thm}

\begin{proof}
Let $q_2(n)$ be the number of 2-colored partitions of $n$.  Recall that $q_2(n)$ has generating function $\prod_{k=1}^\infty \frac{1}{(1 - x^k)^2},$ 
and so in particular we have that $q_2(n) = \sum_{k=0}^n p(k) p(n-k)$.  We recall that an asymptotic for $q_2(n)$ can be given as the following (e.g. see \cite[p.~275]{Murty}): 
\begin{equation*}
        q_2(n) \sim \frac{\sqrt[4]{3} }{12 n^{5/4}} \exp \left(  2 \pi \sqrt{\frac{n}{3}} \right) .
\end{equation*}
We now derive an asymptotic for $T(n) = q_2(n) - q_2(n-2)$.  By the Maclaurin series expansion for $e^{-x}$, note that we have the asymptotic $1 - e^{-x} \sim x$ as $x \to 0$.  By following a similar strategy as done in \cite[Theorem~1]{AGHHPSS}, we can thus show the asymptotic for the ratio $T(n) / q_2(n)$:
\begin{align*}
        \frac{T(n)}{q_2(n)} &= \frac{q_2(n) - q_2(n-2)}{q_2(n)} = 1 - \frac{q_2(n-2)}{ q_2(n) }  \\  
        &\sim 1 - \exp{\left( \frac{5}{4} \log \Big( \frac{n}{n-2} \Big) +  2 \pi \left(  \sqrt{(n-2)/3} - \sqrt{n/3} \right) \right)}.
\end{align*}
Since $1 - e^{x} \sim -x \text{ as } x \to 0,\,\log\!\bigl(\tfrac{n}{n-2}\bigr) = \tfrac{2}{n} + O(n^{-2}), \ \sqrt{n} - \sqrt{n-2} = \tfrac{1}{\sqrt{n}} + O(n^{-3/2}),$ we obtain $\tfrac{T(n)}{q_2(n)}\sim \tfrac{2 \pi }{\sqrt{3n}}$.
Therefore, we have the asymptotic:
    \begin{equation*}
        T(n) \sim \frac{2 \pi }{\sqrt{3n}} q_2(n) \sim \frac{\pi}{6 \sqrt[4]{3} \cdot n^{7/4}} \exp \left(  2 \pi \sqrt{\frac{n}{3}} \right).
    \end{equation*}
\end{proof}

It turns out that there is also another way to prove Theorem \ref{ts1}.

\begin{thm}
If $x \to 1^-,$ then 
$$
G_T(x)\;\sim\;
\frac{(1 - x)^2}{{\pi}}\,
\exp\Bigl(\frac{\pi^2}{3\,(1 - x)}\Bigr).$$
\end{thm}
\begin{proof}
From the equation \eqref{GTP2} we get $G_T(x) = P(x)^2 (1 - x^2)$.\\
From  \cite[section 3.2]{HR},
if $x \to 1^-$, then
$$
(1 - x)\,P(x)
\;\sim\;
\frac{(1 - x)^{3/2}}{\sqrt{2\pi}}\,
\exp\Bigl(\frac{\pi^2}{6\,(1 - x)}\Bigr).
$$
Thus, if $x \to 1^-$, then
$$
G_T(x)\;\sim\;
\frac{(1 - x)^2}{{\pi}}\,
\exp\Bigl(\frac{\pi^2}{3\,(1 - x)}\Bigr).$$
\end{proof}

\begin{re}
Observing that $T(n)$ is weakly increasing and knowing the asymptotic behaviour of the generating function $G_T(x)$ as $x\to1^-$, by applying Ingham’s Tauberian theorem \cite{Ingham1941} we can once again obtain the formula stated in Theorem \ref{ts1}; see also \cite[Theorem 1.1]{BJM2023}.
\end{re}

From Theorem \ref{ts1} and the Hardy–Ramanujan asymptotic for the partition function $p(n)$, we get the following.
\begin{re}
If $n\to \infty,$ then
$\tfrac{\ln T(n)}{\ln p(n)} \sim \sqrt{2}.$
\end{re}

\section{Restricted partitions of a rectangle}

As in the case of the classical partition problems, one can also consider partitions of a rectangle under some additional restrictions. Let us require that only a few types of blocks might be used.

For the sake of clarity, we define $p_{k,l}(2,n)$ as the number of partitions of the rectangle $2\times n$, where only the blocks of sizes $1\times i$ and $2\times j$ are allowed for any $1\leq i\leq k$ and $2\leq j\leq l$. 
For example, if the block of size $1\times1$ is the only one permitted, then it is clear that $p_{1,1}(2,n)=1$. Adding the block $1\times2$ leads us to $p_{2,1}(2,n)=n+1$. One can calculate the explicit formulae of $p_{k,l}(2,n)$ for small values of $k$ and $l$, we collect a few of them in Proposition \ref{proposition: p_k,l}. It is easy to notice that the functions $p_{k,l}(2,n)$ become more and more complicated as the values of $k$ and $l$ grow. Therefore, it might be a non-trivial task to present an explicit formula for $p_{k,l}(2,n)$ in general. However, we can still try to derive the lower and the upper bounds.

The basic lower bound follows from Almkvist's theorem \cite{GA}. To state his result, let us define symmetric polynomials $\sigma_i(x_1,x_2,\ldots,x_k)$ in terms of the power series expansion
\begin{align*}
\sum_{m=0}^\infty\sigma_m(x_1,x_2,\ldots,x_k)t^m:=\prod_{i=1}^k\frac{x_it/2}{\sinh(x_it/2)},
\end{align*} 
and recall that the $A$-partition function $p_A(n)$ counts the number of partitions of $n$ whose parts belong to a fixed multiset $A$ of positive integers. For more information concerning $p_A(n)$, we refer the reader to \cite{GK, GK1} and references therein. Now, we have the following.
\begin{thm}[Almkvist]\label{theorem: Almkvist}
Let $A=\{a_1,a_2,\ldots,a_k\}$ be a fixed multiset of positive integers and put $s_1:=a_1+a_2+\cdots+a_k$. For a given integer $1\leqslant j\leqslant k$, if $\gcd B=1$ for every $j$-element multisubset $B$ of $A$, then
\begin{equation*}
p_A(n)=\frac{1}{\prod_{i=1}^ka_i}\sum_{i=0}^{k-j}\sigma_i(a_1,a_2,\ldots,a_k)\frac{(n+s_1/2)^{k-1-i}}{(k-1-i)!}+O(n^{j-2}).
\end{equation*}
\end{thm}
It might be verified that $\sigma_i=0$ if $i$ is odd. Furthermore, if we set\\$s_m:=a_1^m+a_2^m+\cdots+a_k^m$, then
\begin{align*}
\sigma_0=1\text{,}\hspace{0.2cm}\sigma_2=-\frac{s_2}{24}\text{,}\hspace{0.2cm}\sigma_4=\frac{5s_2^2+2s_4}{5760}\text{,}\hspace{0.2cm}\sigma_6=-\frac{35s_2^3+42s_2s_4+16s_6}{2903040}.
\end{align*}

Now, let us fix parameters $k,l,n\geq1$ with $l\neq1$ and observe that 
\begin{equation*}
    p_{k,l}(2,n)\geq p_A(n),
\end{equation*}
where $A=\{a_1,a_2\ldots,a_{k+l-1}\}=\{1_1,2_1,\ldots,k_1,2_2,3_2,\ldots,l_2\}$ is a multiset that distinguish two numbers if their values are the same. The above inequality follows from the following procedure. Any $A$-partition of $n$ might be represented graphically as a rectangle of size $1\times n$, where blocks of sizes $1\times s_1$ are blank for $1\leq s\leq k$ and blocks of sizes $1\times t_2$ are colored for $2\leq t\leq l$. We double each of the blank rectangles, and change sizes of every colored rectangle from $1\times t_2$ to $2\times t_2$ and make it blank. After that we get a partition of the rectangle $2\times n$. Moreover, two different $A$-partitions of $n$ contribute different rectangle partitions of the block $2\times n$. In conclusion, Theorem \ref{theorem: Almkvist} implies the following.

\begin{pr}\label{proposition: p_k,l from A-partitions}
    Let $k,l,n\geq1$, $l\neq1$ and $A=\{1_1,2_1,\ldots,k_1,2_2,3_2,\ldots,l_2\}$. Then, we have that
\begin{equation*}
    p_{k,l}(2,n)\geq p_A(n)=\frac{n^{k+l-2}}{(k+l-2)!\prod_{i=1}^{k+l-1}a_i}+O(n^{k+l-3}).
\end{equation*}    
\end{pr}

For convenience, we illustrate the above discussed procedure in practice.

\begin{ex}
    Set $A=\{1_1,2_1,2_2,3_2,4_2\}$ and $n=4$. We have $p_A(4)=8$, because all the $A$-partitions of $4$ correspond to
\begin{equation*}
\begin{split}
    &(4_2)\\
    &(2_1,2_1)
\end{split}
\hspace{0.5cm}
\begin{split}
    &(3_2,1_1)\\
    &(2_2,1_1,1_1)
\end{split}
\hspace{0.5cm}
\begin{split}
    &(2_2,2_2)\\
    &(2_1,1_1,1_1)
\end{split}
\hspace{0.5cm}
\begin{split}
    &(2_2,2_1)\\
    &(1_1,1_1,1_1,1_1)
\end{split}
\end{equation*}    
These representations might be interpreted graphically as follows.

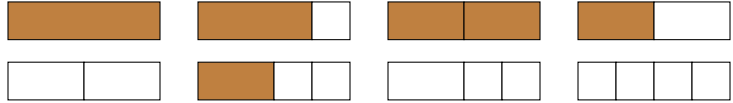
\begin{align*}
     \begin{tikzpicture}[scale=0.5]  
    \def\firstrectangle {(6,6) rectangle (4,4.5)};
    \draw[color=black, fill = brown] (4,4.5) rectangle (8,5.5);
    \draw[color=black, fill = brown] (9,4.5) rectangle (12,5.5);
    \draw[color=black, fill = white] (12,4.5) rectangle (13,5.5);
    \draw[color=black, fill = brown] (14,4.5) rectangle (16,5.5);
    \draw[color=black, fill = brown] (16,4.5) rectangle (18,5.5);
    \draw[color=black, fill = brown] (19,4.5) rectangle (21,5.5);
    \draw[color=black, fill = white] (21,4.5) rectangle (23,5.5);
  \end{tikzpicture}\\
  \begin{tikzpicture}[scale=0.5]  
    \def\firstrectangle {(6,6) rectangle (4,4.5)};
    \draw[color=black, fill = white] (4,4.5) rectangle (6,5.5);
    \draw[color=black, fill = white] (6,4.5) rectangle (8,5.5);
    \draw[color=black, fill = brown] (9,4.5) rectangle (11,5.5);
    \draw[color=black, fill = white] (11,4.5) rectangle (12,5.5);
    \draw[color=black, fill = white] (12,4.5) rectangle (13,5.5);
    \draw[color=black, fill = white] (14,4.5) rectangle (16,5.5);
    \draw[color=black, fill = white] (16,4.5) rectangle (17,5.5);
    \draw[color=black, fill = white] (17,4.5) rectangle (18,5.5);
    \draw[color=black, fill = white] (19,4.5) rectangle (20,5.5);
    \draw[color=black, fill = white] (20,4.5) rectangle (21,5.5);
    \draw[color=black, fill = white] (21,4.5) rectangle (22,5.5);
    \draw[color=black, fill = white] (22,4.5) rectangle (23,5.5);
  \end{tikzpicture}
\end{align*}
Each of the above rectangle is of size $1\times m$ ($1\leq m\leq 4$). Hence, we can proceed as follows. If a rectangle is brown, we make it blank and modify its width from $1$ to $2$. If it is blank, then we double it. In particular, the above configurations correspond to the following ones:

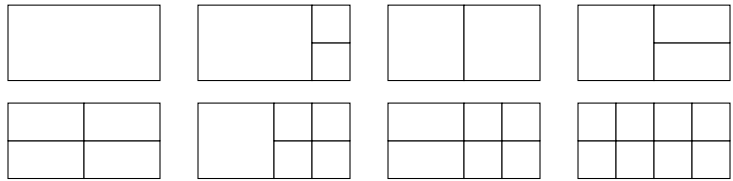
\begin{align*}
     \begin{tikzpicture}[scale=0.5]  
    \def\firstrectangle {(6,6) rectangle (4,4.5)};
    \draw[color=black, fill = white] (4,4.5) rectangle (8,6.5);
    \draw[color=black, fill = white] (9,4.5) rectangle (12,6.5);
    \draw[color=black, fill = white] (12,4.5) rectangle (13,5.5);
    \draw[color=black, fill = white] (12,5.5) rectangle (13,6.5);
    \draw[color=black, fill = white] (14,4.5) rectangle (16,6.5);
    \draw[color=black, fill = white] (16,4.5) rectangle (18,6.5);
    \draw[color=black, fill = white] (19,4.5) rectangle (21,6.5);
    \draw[color=black, fill = white] (21,4.5) rectangle (23,5.5);
    \draw[color=black, fill = white] (21,5.5) rectangle (23,6.5);
  \end{tikzpicture}\\
  \begin{tikzpicture}[scale=0.5]  
    \def\firstrectangle {(6,6) rectangle (4,4.5)};
    \draw[color=black, fill = white] (4,4.5) rectangle (6,5.5);
    \draw[color=black, fill = white] (4,5.5) rectangle (6,6.5);
    \draw[color=black, fill = white] (6,4.5) rectangle (8,5.5);
    \draw[color=black, fill = white] (6,5.5) rectangle (8,6.5);
    \draw[color=black, fill = white] (9,4.5) rectangle (11,6.5);
    \draw[color=black, fill = white] (11,4.5) rectangle (12,5.5);
    \draw[color=black, fill = white] (11,5.5) rectangle (12,6.5);
    \draw[color=black, fill = white] (12,4.5) rectangle (13,5.5);
    \draw[color=black, fill = white] (12,5.5) rectangle (13,6.5);
    \draw[color=black, fill = white] (14,4.5) rectangle (16,5.5);
    \draw[color=black, fill = white] (14,5.5) rectangle (16,6.5);
    \draw[color=black, fill = white] (16,4.5) rectangle (17,5.5);
    \draw[color=black, fill = white] (16,5.5) rectangle (17,6.5);
    \draw[color=black, fill = white] (17,4.5) rectangle (18,5.5);
    \draw[color=black, fill = white] (17,5.5) rectangle (18,6.5);
    \draw[color=black, fill = white] (19,4.5) rectangle (20,5.5);
    \draw[color=black, fill = white] (19,5.5) rectangle (20,6.5);
    \draw[color=black, fill = white] (20,4.5) rectangle (21,5.5);
    \draw[color=black, fill = white] (20,5.5) rectangle (21,6.5);
    \draw[color=black, fill = white] (21,4.5) rectangle (22,5.5);
    \draw[color=black, fill = white] (21,5.5) rectangle (22,6.5);
    \draw[color=black, fill = white] (22,4.5) rectangle (23,5.5);
    \draw[color=black, fill = white] (22,5.5) rectangle (23,6.5);
  \end{tikzpicture}
\end{align*}
That exhibits how the lower bound is derived.
\end{ex}

\begin{re}{\rm
In fact, we could extend the lower bound to any collection of permitted blocks. More precisely, let $B_1,B_2\subset\mathbb{N}$ be fixed and $1\not\in B_2$. Our previous discussion asserts that the number of partitions $p_{B_1,B_2}(2,n)$ of a rectangle of size $2\times n$ that only use blocks of sizes $1\times i$ and $2\times j$ with $i\in B_1$ and $j\in B_2$ is bounded from below by $p_{B_1\cup B_2}(n)$, where $B_1\cup B_2$ is treated as a union of multisets.}
\end{re}

In order to see the usefulness of our lower bound for $p_{k,l}(2,n)$, we present their explicit formulae for small values of $k$ and $l$.

\begin{pr}\label{proposition: p_k,l}
Let us set
\begin{align*}
\begin{split}
    A_{1,2}(n)&:=\begin{cases}
        1 & \text{if }  n\equiv0\pmod{2},  \\
        \frac{1}{2} & \text{if }  n\equiv1\pmod{2},
    \end{cases}\\
    A_{1,3}(n)&:=\begin{cases}
        1 & \text{if }  n\equiv0\pmod{6},  \\
        \frac{5}{12} & \text{if }  n\equiv\pm1\pmod{6}, \\
        \frac{2}{3} & \text{if }  n\equiv\pm2\pmod{6}, \\
        \frac{3}{4} & \text{if }  n\equiv3\pmod{6}, 
    \end{cases}\\
    A_{2,2}(n)&:=\begin{cases}
        1 & \text{if }  n\equiv0\pmod{2},  \\
        \frac{3}{4} & \text{if }  n\equiv1\pmod{2},
    \end{cases}\\
    A_{2,3}(n)&:=\begin{cases}
        1 & \text{if }  n\equiv0\pmod{6},  \\
        \frac{55}{72} & \text{if }  n\equiv1\pmod{6}, \\
        \frac{7}{9} & \text{if }  n\equiv2\pmod{6}, \\
        \frac{7}{8} & \text{if }  n\equiv3\pmod{6}, \\
        \frac{8}{9} & \text{if }  n\equiv4\pmod{6}, \\
        \frac{47}{72} & \text{if }  n\equiv5\pmod{6}, 
    \end{cases}
    \end{split}
    \begin{split}
    A_{3,1}(n)&:=\begin{cases}
        1 & \text{if }  n\equiv0\pmod{3},  \\
        \frac{2}{3} & \text{if }  n\equiv1\pmod{3}, \\
        -\frac{1}{3} & \text{if }  n\equiv2\pmod{3}, 
    \end{cases}\\
    A_{3,2}(n)&:=\begin{cases}
        1 & \text{if }  n\equiv0\pmod{6},  \\
        \frac{25}{36} & \text{if }  n\equiv1\pmod{6}, \\
        \frac{2}{9} & \text{if }  n\equiv2\pmod{6}, \\
        \frac{5}{4} & \text{if }  n\equiv3\pmod{6}, \\
        \frac{4}{9} & \text{if }  n\equiv4\pmod{6}, \\
        \frac{17}{36} & \text{if }  n\equiv5\pmod{6}, 
    \end{cases}\\
    A_{3,3}(n)&:=\begin{cases}
        n+1 & \text{if }  n\equiv0\pmod{6},  \\
        \frac{22n}{27}+\frac{155}{216} & \text{if }  n\equiv1\pmod{6}, \\
        \frac{20n}{27}+\frac{8}{27} & \text{if }  n\equiv2\pmod{6}, \\
        n+\frac{9}{8} & \text{if }  n\equiv3\pmod{6}, \\
        \frac{22n}{27}+\frac{16}{27}& \text{if }  n\equiv4\pmod{6}, \\
        \frac{20n}{27}+\frac{91}{216} & \text{if }  n\equiv5\pmod{6}. 
    \end{cases}
    \end{split}
\end{align*}
Then, the formulae for $p_{k,l}(2,n)$ for $1\leq k,l\leq3$ present as follows: 
    \begin{table}[htb]
\begin{tabular}{ ||c||c|c|c| }
 \hline
 \backslashbox{$k$}{$l$} & $1$ & $2$ & $3$ \\\hline\hline 
 $1$ & $1$ & $\frac{n}{2}+A_{1,2}(n)$ & $\frac{n^2}{12}+\frac{n}{2}+A_{1,3}(n)$\\\hline 
 $2$ & $n+1$ & $\frac{n^2}{4}+n+A_{2,2}(n)$ & $\frac{n^3}{36}+\frac{7n^2}{24}+\frac{11n}{12}+A_{2,3}(n)$\\\hline 
 $3$ & $\frac{n^2}{3}+n+A_{3,1}(n)$ & $\frac{n^3}{18}+\frac{5n^2}{12}+\frac{5n}{6}+A_{3,2}(n)$ & $\frac{n^4}{216}+\frac{2n^3}{27}+\frac{7n^2}{18}+A_{3,3}(n)$\\
 \hline
\end{tabular}
\medskip
\caption{The formulae for $p_{k,l}(2,n)$ for $1\leq k,l\leq3$.}
\end{table}
\end{pr}

\begin{proof}
    The formulae for $p_{1,1}(2,n)=1$ and $p_{2,1}(2,n)=n+1$ are clear. Further, both of the functions $p_{1,2}(2,n)$ and $p_{1,3}(2,n)$ correspond directly to the unique $A$-partition functions: $p_{\{1,2\}}(n)$ and $p_{\{1,2,3\}}(n)$, respectively. Indeed, it is not difficult to observe that the corresponding functions satisfy the same recurrence relation with the identical initial values. Hence, we have that $p_{1,2}(2,n)=p_{\{1,2\}}(n)$ and $p_{1,3}(2,n)=p_{\{1,2,3\}}(n)$, but the exact values of both $p_{\{1,2\}}(n)$ and $p_{\{1,2,3\}}(n)$ are known and might be found, for instance, in Andrews and Eriksson's book, see \cite[Chapter 6]{GA1}.
    
    In order to show the equality for $p_{3,1}(2,n)$, let us first find out the appropriate recurrence relation. We can either take the block $1\times3$ as a part or not. If we do not, then we consider the value of $p_{2,1}(2,n)=n+1$. Otherwise, we use the block $1\times3$ either exactly one time or at least two times. 

In the first case, we put the the block $1\times 3$ in the bottom right corner and need to fill the remaining strip of the rectangle $2\times n$ using only blocks $1\times1$ or $1\times2$. Since we can assume that the blocks $1\times2$ are put only horizontally (in that case), there are exactly $\floor{\frac{n-3}{2}}+\floor{\frac{n}{2}}+1=n-1$ possibilities to fill the rectangle $2\times n$ in such a setting. 

On the other hand, if the block $1\times3$ occurs at least two times, then we can require that at the end of the rectangle $2\times n$ there are two blocks $1\times3$ (one on the top of the other). To see that, let us consider all possible ways to fill the rectangle of size $2\times 6$ with blocks of sizes $1\times1$, $1\times2$ and $1\times3$ in such a way that there are exactly two blocks of size $1\times 3$ in one row. If we do that, then it suffices to notice that each of a configuration may be rearranged in such a way that there are two blocks of size $1\times3$ one on the top of the other. In consequence, it implies that we have exactly $p_{3,1}(2,n-3)$ possibilities to fill the reaming strip of the original rectangle.

Hence, we get that $p_{3,1}(2,0)=1$, $p_{3,1}(2,1)=2$, $p_{3,1}(2,2)=3$ and
\begin{align*}
p_{3,1}(2,n)=p_{3,1}(2,n-3)+2n
\end{align*}
for $n\geq3$. After some elementary computations we obtain the required formula. 

Finally, to deduce the remaining formulae from Table $1$, let us notice that we have the following recurrence relations
\begin{align}\label{equation: p_k,l}
p_{k,l}(2,n)=p_{k,l}(2,n-l)+p_{k,l-1}(2,n)
\end{align}
for $k,l\geq2$ and $n\geq l$. They follow from the fact that we can either take the block $2\times l$ as a part of our partition or not. If we do, then it is enough to calculate the value of $p_{k,l}(2,n-l)$. If not, then we have exactly $p_{k,l-1}(2,n)$ possibilities to fill the rectangle $2\times n$. 

To finish the proof, it is enough to systematically apply the equality \eqref{equation: p_k,l}. After elementary but tedious computations, we obtain the required formulae. 
\end{proof}

There are some consequences and remarks concerning both Proposition \ref{proposition: p_k,l from A-partitions} and Proposition \ref{proposition: p_k,l}. Sometimes, it is not difficult to determine the exact formula for a  given restricted rectangular partition function.

\begin{re}\label{remark: 2 parts}{\rm
    Let $m\geq3$ be fixed. For the set $B=\{1,m\}$, one can easily derive the formula for $p_{B,\emptyset}(2,n)$ (the number of partitions of the rectangle $2\times n$, where only block of sizes $1\times1$ and $1\times m$ are permitted). We have that $p_{B,\emptyset}(2,n)=2\floor{\frac{n}{m}}+1$. 
}\end{re}
Further, we can also have a direct correspondence between a considered restricted rectangular partition function and an appropriate $A$-partition function.

\begin{re}\label{remark: p_1,l}{\rm
    For a fixed $l\geq2$, let us denote by $p_l(n)$ the number of partitions of $n$ into parts smaller or equal than $l$. In such a setting, it turns out that
    \begin{equation*}
        p_{1,l}(2,n)=p_{l}(n).
    \end{equation*}
    That is a direct consequence of the fact that considering $p_{1,l}(n)$, we may always treat block $1\times1$ as a block of size $2\times1$ (in a vertical position). Thus, $p_{1,l}(2,n)$ enumerates the number of partition of a rectangle of size $2\times n$ into blocks of sizes $2\times1,2\times2,\ldots,2\times l$. Dividing their width by $2$ leads us directly to partitions of $n$ with the largest part at most $l$, as required.
    
    It is worth saying that the exact formulae for $p_{l}(n)$ for $1\leq l\leq60$ are collected in Sills and Zeilberger's paper, see \cite{SZ}. Furthermore, the above discussion asserts that the lower bound obtained in Proposition \ref{proposition: p_k,l from A-partitions} is optimal in the case of $p_{1,l}(2,n)$.
}\end{re}

As it was mentioned in Remark \ref{remark: p_1,l}, Proposition \ref{proposition: p_k,l from A-partitions} is optimal for $p_{1,l}(2,n)$, but it is not the case in general, as Proposition \ref{proposition: p_k,l} indicates. In fact, it seems that for $l\geq1$ the function $p_{k,l}(2,n)$ grows as fast as a polynomial of degree $k+l-2$, but there is a problem with the leading coefficient stated in Proposition \ref{proposition: p_k,l from A-partitions}. To find the appropriate leading term, we start by investigating the asymptotic behavior of $p_{k,1}(2,n)$. Why? Because, for $l\geq2$, the asymptotic of $p_{k,l}(2,n)$ can be later derived by utilizing the equality \eqref{equation: p_k,l}. Therefore, we would like to extend the reasoning exhibited in the proof of Proposition \ref{proposition: p_k,l} in the case of $p_{3,1}(2,n)$ to derive the asymptotics of $p_{k,1}(2,n)$ in general. We wish to require that the largest blocks occur at the end of the rectangle $2\times n$, and if there are at least two blocks of size $1\times k$, then we put them one on the top of the other. However, sometimes it can not be done. Even for $p_{4,1}(2,n)$, there appear some problematic cases. For instance, there is exactly one possibility (up to the swap roles of the rows) to partition the rectangle of size $2\times12$ into $3$ blocks of size $1\times4$ and $4$ blocks of size $1\times 3$, as Figure $1$ exhibits.

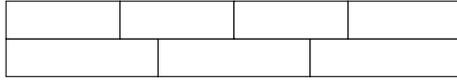
\begin{figure}[!htb]\label{Figure 1}
 \begin{tikzpicture}[scale=0.5]  
    \def\firstrectangle {(6,6) rectangle (4,4.5)};
    \draw[color=black, fill = white] (8,4.5) rectangle (12,5.5);
    \draw[color=black, fill = white] (4,4.5) rectangle (8,5.5);
    \draw[color=black, fill = white] (0,4.5) rectangle (4,5.5);
    \draw[color=black, fill = white] (0,5.5) rectangle (3,6.5);
    \draw[color=black, fill = white] (3,5.5) rectangle (6,6.5);
    \draw[color=black, fill = white] (6,5.5) rectangle (9,6.5);
    \draw[color=black, fill = white] (9,5.5) rectangle (12,6.5);
  \end{tikzpicture}
   \caption{A partition of the rectangle $2\times 12$ into $4$ blocks $1\times3$ and $3$ blocks $1\times4$.}
\end{figure}
Nevertheless, it turns out that the situations, like the above one, are rare, and they do not essentially have an impact on the asymptotic of $p_{k,1}(2,n)$. 

\begin{pr}\label{proposition: asymptotics of p_k,0}
    Let $k\geq1$. For every $n\geq1$, we have that
\begin{align*}
        \frac{(2n)^{k-1}}{k!(k-1)!}-a_kn^{k-2}\leq p_{k,1}(2,n)\leq \frac{(2n)^{k-1}}{k!(k-1)!}+a_kn^{k-2},
    \end{align*}
    where $a_k>0$ is some constant independent of $n$.
\end{pr}
\begin{proof}
    Let us proceed by the induction on $k$. Proposition \ref{proposition: p_k,l} asserts that the statement is valid for $p_{1,1}(2,n)$, $p_{2,1}(2,n)$ and $p_{3,1}(2,n)$. Hence, we assume the correctness of the claim for all values between $2$ and $k-1$, and check its validity for $p_{k,1}(2,n)$. Let $n\geq1$ be fixed. Considering all partitions of the rectangle $2\times n$ into blocks of sizes $1\times1,1\times 2,\ldots,1\times k$, we have several possibilities.
\begin{case}
    There is no block of size $1\times k$. If that is the case, then we have exactly $p_{k-1,1}(2,n)$ ways to partition the rectangle.
\end{case}
\begin{case}
    There are more than one block of size $1\times k$, and two of them can be placed at the end of the rectangle $2\times n$ (one on the top of the other). This case contributes $p_{k,1}(2,n-k)$ partitions of the rectangle.
\end{case}
\begin{case}
    There is exactly one block of size $1\times k$ located in the south-east corner of the rectangle $2\times n$. Let us denote the number of all partitions enumerated in this case by $A(n)$. Then, it is easy to see that $p_{k-1,1}(2,n-k)<A(n)<p_{k-1,1}(2,n)$.
\end{case}
\begin{case}
    There are at least two blocks of size $1\times k$ that can not be arranged as in Case $2$. Let us denote the number of all partitions enumerated in this case by $B(n)$. At first, we notice that the pigeonhole principle asserts that the  number of occurrences of the block $1\times k$ is at most $k(k-1)^2/2$. To see this, let us suppose that there are $k(k-1)^2/2+1$ blocks of size $1\times k$ in the bottom row. However, if we assume that each block $1\times j$ for $j\in\{1,2,\ldots,k-1\}$ occurs $k-1$ times in the top row, then their total length is $k(k-1)^2/2$. Hence, one of them, let say $1\times i$, needs to appear $k$ times. But, it implies that we may swap $k$ blocks $1\times i$ with $i$ blocks $1\times k$, shift these blocks $1\times k$ to the end of the rectangle $2\times n$, and get Case $2$. This is a contradiction.
    
    Further, without loss of generality, we can require that all blocks $1\times k$ are located in the bottom row. The above discussion, in particular, implies that the blocks $1\times1$ and $1\times k-1$ can not occur at the same time in the top row. Otherwise, we could just replace them with one block of size $1\times k$, shift the block to the end, and get Case $2$. Therefore, we have the following possibilities:
    \begin{enumerate}
        \item the block $1\times k-1$ appears in the top row and the block $1\times 1$ does not occur there;
        \item the block $1\times 1$ appears in the top row and the block $1\times k-1$ does not occur there;
        \item neither $1\times k-1$ nor $1\times 1$ appears as a part in the top row.
    \end{enumerate}Let us denote the contribution of these cases to $B(n)$ by $B_1(n)$, $B_2(n)$ and $B_3(n)$, respectively.

    If the block $1\times k-1$ appears in the top row and the block $1\times 1$ does not occur there, then we have at most $k-2$ blocks $1\times 1$ in the bottom row. Otherwise, $k-1$ blocks of size $1\times1$ can be  swapped with the block $1\times k-1$. We also know that there is another block in the top row of size $1\times j$, where $j\in\{2,3,\ldots,k-1\}$. Thus, we may collect $k-j$ blocks of size $1\times1$ and the block $1\times j$, and swap them with the block $1\times k$. If we do that, then it is enough to shift the chosen block $1\times k$ to the end of the rectangle $2\times n$ to get Case $2$, and the contradiction in the same time. Thus, let us suppose that the block $1\times1$ occurs $i$ times for some $i\in\{0,1,\ldots,k-2\}$. If that is the case, then we can only use the blocks $1\times 2,1\times3,\ldots,1\times k-1$, to fill the remaining strip of the rectangle $2\times n$. However, the total number of such tillings is always smaller than the total number of tillings of the same strip using the blocks $1\times1,1\times2,\ldots,1\times k-2$. Hence, we can easily deduce that 
    \begin{align*}
        B_1(n)<(k-1)p_{k-2,1}(2,n),
    \end{align*}
where $k-1$ stands for the number of possible appearances of the block $1\times1$.

    If the block $1\times 1$ appears in the top row and the block $1\times k-1$ does not occur there, then we have at most $k-1$ blocks $1\times 1$ in the top row. Otherwise, we can simply replace them with the block $1\times k$, and obtain Case $2$ by the same procedure as before. Analogous reasoning to that one from the previous paragraph points out that
    \begin{align*}
        B_2(n)<kp_{k-2,1}(2,n).
    \end{align*}

    Finally, if neither $1\times k-1$ nor $1\times 1$ appears as a part in the top row, then the block $1\times1$ can occur at most $k-1$ times in the bottom row. Otherwise, we could systematically swap the blocks from the top row (starting from the longest ones) with the appropriate number of blocks $1\times 1$ from the bottom row as long as possible. At the end, we will obtain $j$ blocks of size $1\times1$ for some $j\geq0$ in the bottom row that can not be swapped with any block from the top row. However, it means that only blocks of lengths grater than $j$ remain in the top row. If we take one of them, let say $1\times i$, we may marge it with $k-i$ ($k-i<k-j$) blocks $1\times 1$ from the top row, and replace by the block $1\times k$ from the bottom row to get Case $2$. This contradiction together with the reasoning from the penultimate paragraph gives us that
    \begin{align*}
        B_3(n)<kp_{k-2,1}(2,n).
    \end{align*}

    Hence, we may summarize Case $4$ by noting that
    \begin{align*}
        B(n)\leq \frac{k(k-1)^2}{2}\left(B_1(n)+B_2(n)+B_3(n)\right)<k^5p_{k-2,1}(2,n).
    \end{align*}
\end{case}

All of the discussion presented above leads us to both the lower bound:
\begin{align}\label{ine: lower-bound}
    p_{k,1}(2,n)\geq p_{k,1}(2,n-k)+p_{k-1,1}(2,n)+p_{k-1,1}(2,n-k)
\end{align}
and the upper bound:
\begin{align}\label{ine: upper-bound}
    p_{k,1}(2,n)\leq p_{k,1}(2,n-k)+2p_{k-1,1}(2,n)+k^5p_{k-2,1}(2,n).
\end{align}

Now, let us focus on the lower bound for $p_{k,1}(2,n)$. Applying the inequality (\ref{ine: lower-bound}) several times together with the induction hypothesis maintains that
\begin{align*}
    p_{k,1}(2,n)&\geq p_{k-1,1}(2,n)+2p_{k-1,1}(2,n-k)+\cdots+2p_{k-1,1}\left(2,n-\floor{\frac{n}{k}}k\right)\\
    &\geq an^{k-2}-a_{k-1}n^{k-3}+2\sum_{j=1}^{\floor{n/k}}\left(a(n-jk)^{k-2}-a_{k-1}(n-jk)^{k-3}\right),
\end{align*}
where $a=2^{k-2}/((k-1)!(k-2)!)$. Further, for any $s\geq0$ we have that
\begin{align*}
    \sum_{j=1}^{\floor{n/k}}(n-jk)^s&>\int_{0}^{\floor{n/k}}(n-(x+1)k)^sdx\\
    &=-\frac{1}{k(s+1)}\left(\left(n-(\floor{n/k}+1)k\right)^{s+1}-\left(n+k\right)^{s+1}\right)\\
    &\geq\frac{n^{s+1}}{k(s+1)}-\frac{k^s}{s+1}.
\end{align*}
Therefore, we deduce that
\begin{align*}
     p_{k,1}(2,n)&>an^{k-2}-a_{k-1}n^{k-3}+\frac{2an^{k-1}}{k(k-1)}-\frac{2ak^{k-2}}{k-1}-2a_{k-1}n^{k-2}\\
     &>\frac{2an^{k-1}}{k(k-1)}-wn^{k-2}=\frac{(2n)^{k-1}}{k!(k-1)!}-wn^{k-2},
\end{align*}
where $w=3a_{k-1}+2ak^{k-2}/(k-1)$.

In the case of the upper bound, the inequality (\ref{ine: upper-bound}) together with the induction hypothesis guarantees that
\begin{align*}
     p_{k,1}(2,n)&\leq 2p_{k-1,1}(2,n)+k^5p_{k-2,1}(2,n)+p_{k,1}(2,n-k)\\
     &\leq 2\sum_{j=0}^{\floor{n/k}}p_{k-1,1}(2,n-jk)+k^5\sum_{j=0}^{\floor{n/k}}p_{k-2,1}(2,n-jk)\\
     &\leq 2\sum_{j=0}^{\floor{n/k}}\left(a(n-jk)^{k-2}+a_{k-1}(n+jk)^{k-3}\right)+k^5n\left(bn^{k-3}+a_{k-2}n^{k-4}\right),
\end{align*}
where $b=2^{k-3}/((k-2)!(k-3)!)$. Moreover, for any $s\geq0$ we have that
\begin{align*}
    \sum_{j=1}^{\floor{n/k}}(n-jk)^s&<\int_{0}^{\floor{n/k}}(n-xk)^sdx\\
    &=-\frac{1}{k(s+1)}\left(\left(n-\floor{n/k}k\right)^{s+1}-n^{s+1}\right)\leq\frac{n^{s+1}}{k(s+1)}.
\end{align*}
Thus, it follows that
\begin{align*}
    p_{k,1}(2,n)&\leq 2an^{k-2}+a_{k-1}n^{k-3}+\frac{2an^{k-1}}{k(k-1)}+\frac{2a_{k-1}n^{k-2}}{k(k-2)}+(bn+a_{k-2})k^5n^{k-3}\\
    &\leq \frac{2an^{k-1}}{k(k-1)}+zn^{k-2}=\frac{(2n)^{k-1}}{k!(k-1)!}+zn^{k-2},
\end{align*}
where $z=2a+a_{k-1}+2a_{k-1}/(k(k-2))+(b+a_{k-2})k^5$. At the moment, it suffices to set $a_k:=\max\{w,z\}$. This ends the induction step and thereby the proof.
\end{proof}

We may extend Proposition \ref{proposition: asymptotics of p_k,0} to obtain the lower and the upper bounds for $p_{k,l}(2,n)$, when $l\geq1$.

\begin{pr}\label{proposition: asymptotics of p_k,l}
    Let $k,l\geq1$ be fixed. For every $n\geq1$, we have that
    \begin{align*}
        \frac{2^{k-1}n^{k+l-2}}{l!k!(k+l-2)!}-c_{k,l}n^{k+l-3}\leq p_{k,l}(2,n)\leq\frac{2^{k-1}n^{k+l-2}}{l!k!(k+l-2)!}+c_{k,l}n^{k+l-3},
    \end{align*}
    where $c_{k,l}>0$ is some constant independent of $n$.
\end{pr}
\begin{proof}
    Let us proceed by the induction on $l$. Proposition \ref{proposition: asymptotics of p_k,0} guarantees the correctness of the statement for $l=1$. Hence, let us demand that the claim is valid for all the functions $p_{k,1}(2,n),p_{k,2}(2,n),\ldots, p_{k,l-1}(2,n)$ and check its correctness for $p_{k,l}(2,n)$. Now, Proposition \ref{proposition: asymptotics of p_k,0} together with its proof and the equality \eqref{equation: p_k,l} implies that
    \begin{align*}
        p_{k,l}(2,n)&\geq\sum_{j=0}^{\floor{n/l}}\left( d(n-jl)^{k+l-3}-c_{k,l-1}(n-jl)^{k+l-4}\right)\\
        &>\frac{dn^{k+l-2}}{l(k+l-2)}-\frac{l^{k+l-3}}{k+l-2}-c_{k,l-1}n^{k+l-3}
    \end{align*}
    and
    \begin{align*}
        p_{k,l}(2,n)&\leq\sum_{j=0}^{\floor{n/l}}\left( d(n-jl)^{k+l-3}+c_{k,l-1}(n-jl)^{k+l-4}\right)\\
        &<\frac{dn^{k+l-2}}{l(k+l-2)}+dn^{k+l-3}+c_{k,l-1}n^{k+l-3},
    \end{align*}
    where $d=2^{k-1}/((l-1)!k!(k+l-3)!)$. Now, it is not difficult to see that the required property holds. This ends both the induction step and the proof.

\end{proof}

Combining Remark \ref{remark: p_1,l}, Proposition \ref{proposition: asymptotics of p_k,0} and Proposition \ref{proposition: asymptotics of p_k,l} completes the proof of one of the main results from Introduction, i.e. Theorem \ref{thm: asymptotics p_k,l}.


There is a natural question arising from Theorem \ref{thm: asymptotics p_k,l}, namely, what can one say about the asymptotic of the number $p_{A,B}(2,n)$ of partitions of the rectangle $2\times n$, when only blocks of sizes $1\times a$ and $2\times b$ are permitted, where $a\in A$, $b\in B$ and $A$ and $B$ are some fixed sets of positive integers. The above discussion might suggest that it should behave in a similar way to $p_{k,l}(2,n)$, but it requires additional investigation. 

At this point, it is worth pointing out that one can also investigate the log-behavior of both $p(2,n)$ and $p_{k,l}(2,n)$. It would be interesting to examine both the log-concavity property and the Bessenrodt-Ono type inequality of these functions. Let us recall that a sequence of real numbers $(c_n)_{n\in\mathbb{N}}$ is said to be log-concave at $n_0$ if $c_{n_0}^2>c_{n_0-1}c_{n_0+1}$. It is well-known that the partition function is log-concave for each $n\geq26$, see \cite{DSP, N}. On the other hand, the Bessenrodt-Ono inequality \cite{B-O} states that
\begin{align*}
    p(a)p(b)>p(a+b)
\end{align*}
is valid for all positive integers $a,b\geq2$ such that $a+b>9$. In fact, the analogue of this property for $p_{k,l}(2,n)$ follows directly from Proposition \ref{proposition: asymptotics of p_k,0}, Proposition \ref{proposition: asymptotics of p_k,l} and the proof of Theorem 5.4 from \cite{GK1}. However, we believe that this phenomenon can be established in a more general setting. Therefore, we leave it as an open problem for the interested reader.

\subsection{Computations and conjectures regarding  \texorpdfstring{$p_{k,l}(2, n)$}{p{k,l}(2, n)}}
Recall the formulae for $p_{k,l}(2,n)$ given in Proposition \ref{proposition: p_k,l} for $0 \leq k, l \leq 3$.  While proving similar explicit formulae for larger $k, l$ seems to be out of reach, we can resort to giving various conjectures on the values for $p_{k,l}(2,n)$ for  larger values of $k$ and $l$.  

We recall that a function $f : \mathbb{N}_0 \to \mathbb{N}_0$ is a \emph{quasi-polynomial} if there exists some integer $s \geq 1$ and polynomials $q_0, q_1, \dots, q_{s-1}$ such that $f(n) = q_i(n)$ if $n \equiv i$ (mod $s$).  The \emph{degree} of $f$ is the maximum of the degrees of $q_0, \dots, q_{s-1}$, and the \emph{quasi-period} of $f$ is the smallest such $s \geq 1$ that can be chosen (see \cite[Section~4.4]{Stanley_vol1}).
In particular, for the case $l = 0$, we conjecture the following.

\begin{con} \label{con:quasipoly}
    For any fixed integer $k \geq 1$, the function $p_{k,1}(2,n)$ is a quasi-polynomial in $n$ for sufficiently large $n$, with degree $k-1$ and with a quasi-period dividing $\mathrm{lcm}(1$, $2$, $\ldots$, $k)$.  In particular,  $p_{k,1}(2, n)$ has a rational generating function of the form $P(x)/Q(x)$ where $P(x) \in \Z[x]$, and where the denominator $Q(x) \in \Z[x]$ is a divisor of $\prod_{i=1}^{k} (1 - x^i)$.
\end{con}

By computing the first few hundred terms of $p_{k,1}(2,n)$ for each of $k = 4, 5, \ldots, 8$, we provide a quasi-polynomial conjectured formula for $p_{k,1}(2, n)$ using a similar strategy as done in Sills--Zeilberger \cite{SZ}. Explicitly, we conjecture that $p_{k,1}(2, n)$ is a quasi-polynomial of the form 
\begin{equation*}
    \sum_{j=1}^k  \sum_{i=0}^{k-1} [r_{i,j,0}, r_{i,j,1}, \dots, r_{i,j,j-1}]_j   n^i = p_{k,1}(2, n)  
\end{equation*}
for all sufficiently large $n$.  Here, we use the notation $\left[ r_0, r_1, \dots, r_{m-1} \right]_m$ to denote a function which returns $r_i$ if $n \equiv i$ (mod $m$), where $r_{i,j,k}$ are the variables to be solved. For a suitable $A$ and $B$, we thus set up a system of linear equations of the form
\begin{align*}
    \sum_{j=1}^k  \sum_{i=0}^{k-1} r_{i,j,n (\text{mod } j)} n^i = p_{k,1}(2, n)  \quad \text{ for all } n = A, A+1, \dots, B,
\end{align*}
where $n^i$ are the coefficients of our linear system, and $p_{k,1}(2, n)$ are the constant terms.  By starting from $A = 1$, we checked with Sage \cite{Sage} if the above linear system has a solution for $r_{i,j,k} \in \Q$.  If no solution was found, we instead start from $A = 2$, again checking if a solution exists.  We repeat the process, incrementing $A$ until a rational solution to the linear system is obtained.

By applying this procedure for each $k = 4, 5, \dots, 8$, we obtain the conjectured quasi-polynomial formulae for $p_{k,1}(2, n)$, as shown in Table~\ref{tab:conjectured_p2n}. 

\begin{table}[h]
\centering
\caption{Our conjectured quasi-polynomial formulae for $p_{k,1}(2, n)$ for all $n \geq N_k$. We use the notation $[r_0, \dots, r_{m-1}]_m$ to denote a function which returns $r_i$ if $n \equiv i$ (mod $m$).}
\label{tab:conjectured_p2n}
\begin{tabular}{@{}ccc@{}}
\toprule
$k$ & Conjectured $p_{k,1}(2,n)$ for all $n \geq N_k$ & $N_k$ \\ \midrule

4  &   $\frac{1}{18} n^3 + \frac{5}{12} n^2 + n + \left[ -\frac{5}{4}, -\frac{53}{36}, - \frac{121}{36} \right]_3 + \left[ \frac{9}{4}, 1, \frac{1}{4}, 0 \right]_4$  & 6 \\[6mm]

5  & \begin{tabular}[c]{@{}c@{}} $\frac{1}{180} n^4  + \frac{1}{12} n^3 + \frac{31}{72} n^2 + \frac{11}{12} n - \frac{3037}{360} + \left[ \frac{19}{9}, 1, 0 \right]_3$ \\ $ +   \left[ \frac{33}{8}, 3, \frac{9}{8}, 0 \right]_4   + \left[ \frac{16}{5}, 1, -1, 0, 0 \right]_5$  \end{tabular}   &  12  \\ [8mm]

6  &  \begin{tabular}[c]{@{}c@{}}  $\frac{1}{2700} n^5  + \frac{7}{720} n^4 + \frac{77}{810} n^3 + \frac{31}{72} n^2 + \frac{31}{540} n +  \left(  \left[ \frac{1}{6}, 0 \right]_2 + \left[ \frac{19}{27}, \frac{1}{3}, 0 \right]_3  \right) n$  \\ 
$ - \frac{430513}{32400}  + \left[ \frac{4169}{1296}, 3, \frac{281}{1296}, 0 \right]_4 + \left[ \frac{178}{25}, \frac{126}{25}, \frac{24}{25}, \frac{27}{25}, 0 \right]_5 $ \\ $ + \left[ \frac{320}{81}, \frac{107}{81}, \frac{1}{81}, -\frac{194}{81}, 0, 0 \right]_6   $   \end{tabular} &   20  \\[10mm] 

7 &  \begin{tabular}[c]{@{}c@{}}   $\frac{1}{56700} n^6 + \frac{1}{1350} n^5 + \frac{79}{6480} n^4 + \frac{161}{1620} n^3 + \frac{251 }{600} n^2 
- \frac{497}{1620} n $ \\ $+ \left( \left[ \frac{1}{6}, 0 \right]_2   +  \left[ \frac{82}{81}, \frac{53}{81}, 0   \right]_3 \right) n 
- \frac{2542973}{75600} +  \left[ \frac{3337}{1296}, 1,  -\frac{3143}{1296},  0 \right]_4 $ \\ $  + \left[ \frac{252}{25},  \frac{152}{25}, 4, \frac{51}{25}, 0 \right]_5  + \left[ \frac{797}{81},  \frac{541}{81},  \frac{377}{81},  -\frac{58}{27}, 0, 0 \right]_6 $ \\  $ + \left[ \frac{85}{7}, 3, -3, 0, 0, 0, 0  \right]_7 $  \end{tabular}    & 30 \\[12mm]

8  &   \begin{tabular}[c]{@{}c@{}}   $\frac{1}{1587600} n^7 + \frac{1}{25200} n^6 + \frac{307}{302400} n^5 + \frac{13}{960} n^4 + \frac{45641}{453600} n^3 + \frac{313}{800} n^2 + \left[  \frac{1}{48}, 0  \right]_2 n^2    $ \\  $ - \frac{46519}{36288} n +  \left( \left[ \frac{82}{81}, \frac{1}{3}, 0 \right]_3   +  \left[ \frac{73}{64}, \frac{1}{4}, - \frac{7}{64}, 0   \right]_4 \right) n - \frac{397440641}{6350400}  $\\  
$ +  \left[ \frac{326}{25}, 7,  \frac{49}{25},  5, 0 \right]_5  + \left[ \frac{76817}{5184},  \frac{949}{81}, \frac{1859}{192}, \frac{65}{81}, \frac{17617}{5184}, 0 \right]_6 $ \\  $ + \left[\frac{1230}{49},  \frac{787}{49},  \frac{197}{49},  \frac{195}{49}, \frac{102}{49}, \frac{51}{49}, 0 \right]_7  + \left[ \frac{85}{8}, 5, -3, 0, -\frac{83}{8}, -1, 0, 0  \right]_8 $  \end{tabular}    & 42  \\
\bottomrule
\end{tabular}
\end{table}
\begin{re}
\rm{As expected, we note that the degree and leading term for each conjectured quasi-polynomial for $p_{k,1}(2, n)$ agrees with the asymptotic proven in Proposition~\ref{proposition: asymptotics of p_k,0}.  Furthermore, it's worth comparing our conjectures with the known values for the standard restricted function $p_k(2n)$; e.g. as given in \cite[p.~641]{SZ}.  From Table \ref{tab:conjectured_p2n}, we note that we actually have $|p_{k,1}(2, n) - p_k(2n)| = O(1)$  for $k \leq 5$, $|p_{k,1}(2, n) - p_k(2n)| = O(n)$ for $k = 6, 7$, and $|p_{k,1}(2, n) - p_k(2n)| = O(n^2)$ for $k = 8$. In general, it seems an interesting problem to compute a tight upper bound for $|p_{k,1}(2, n) - p_k(2n)|$ for all $k$.

We should stress that $p_{k,1}(2, n)$ is not strictly speaking a quasi-polynomial for all $n$, as the above formula only appears to hold for sufficiently large $n$.  Equivalently, while $p_{k,1}(2, n)$ is conjectured to have a rational generating function $P(x)/Q(x)$, the degree of $P(x)$ may in general be larger than the degree of $Q(x)$.}
\end{re}
\begin{pr}
    Let $k, l \geq 2$ be fixed integers and assume Conjecture~\ref{con:quasipoly}.  Then $p_{k,l}(2,n)$ is a quasi-polynomial in $n$ for sufficiently large $n$, with degree $k + l - 2$, and with a quasi-period dividing $\lcm(1, 2, \dots, \max(k, l))$.  In particular, $p_{k,l}(2,n)$ has a rational generating function of the form $P(x)/Q(x)$ where $P(x) \in \Z[x]$, and where the denominator $Q(x) \in \Z[x]$ is a divisor of $\prod_{i=1}^{k} (1 - x^i) \cdot \prod_{i=2}^{l} (1 - x^i)$.
\end{pr}

\begin{proof}
    By applying a similar argument as done in Theorem~\ref{thm: asymptotic p(2,n)}, we recall that any partition of a $2 \times n$ rectangle can be rearranged such that all blocks of type $2 \times j$ for $j \geq 2$ are shifted to the left side of the rectangle.  Thus,  given any such restricted partition of $2 \times n$, it can be rearranged such that there exists a unique $i = 0, \dots, n$, where the first $i$ columns of the rectangle gives a partition of $2 \times i$ using only blocks of type $2 \times j$ for $j \in \{2, 3, \dots, l\}$, and the remaining $n-i$ columns of the rectangle gives a partition of $2 \times (n-i)$  using only blocks of type $1 \times j$ for $j \in \{1, 2, \dots, k\}$.
    This gives the following identity for $p_{k,l}(2,n)$,
    \begin{equation} \label{eq:pkl2n}
        p_{k,l}(2,n) = \sum_{i=0}^n p_{0,B}(2,i) p_{k,1}(2,n -i) ,
    \end{equation}
    where $B = \{2, 3, \dots, l\}$. We observe that $p_{0,B}(2,i)$ is the same as the usual restricted partition function $p_B(i)$ and therefore has generating function $\prod_{i=2}^l (1 - x^i)^{-1}$ and is thus a quasi-polynomial of degree $l-2$ and has quasi-period dividing $\lcm(2, 3, \dots, l)$ (e.g. see \cite[p.~641]{SZ}).

    Therefore, by using the convolution identity in (\ref{eq:pkl2n}), the generating function of $p_{k,l}(2,n)$ is simply the product of the generating functions of $p_B(n)$ and $p_{k,1}(2,n)$.  Therefore, assuming Conjecture \ref{con:quasipoly} that $p_{k,1}(2,n)$ is a quasi-polynomial for sufficiently large $n$, we have that $p_{k,l}(2,n)$ s a quasi-polynomial for sufficiently large $n$, using \cite[Proposition~4.4.41(ii)]{Stanley_vol1}.  Finally, Proposition~\ref{proposition: asymptotics of p_k,l} shows that the degree of $p_{k,l}(2,n)$ is $k+l-2$, as claimed.
\end{proof}


\section{Generalization of the {\it m}-ary partition function}

Sometimes, it turns out that the generalization of a given partition statistic to the rectangular case may behave regularly, in a sense that, one can find an explicit recurrence relation for its values. That is the case of the main object of our interest in this section, i.e. the extension of the so-called $m$-ary partition function.

Let us recall that, for a given parameter $m\geq2$, the $m$-ary partition function $b_m(n)$ may be defined in terms of an $A$-partition function as $b_m(n):=p_A(n)$ with $A:=\{m^i:i\in\mathbb{N}_0\}$. Before we pass to the main part of this section, let us recall some basic properties of the $m$-ary partition function.

\begin{lm}\label{eq: b_m (1)}
Let $m\geq2$ be fixed. For every positive integer $n$ such that $m\nmid n$, we have that 
\begin{align*}
        b_m(n)=b_m(n-1).
    \end{align*}
\end{lm}
\begin{proof}
The equality follows from the fact that in order to get an $m$-ary partition of $n$ we need to add the part $1$ to some $m$-ary partition of $n-1$, and this procedure may be reversed.
\end{proof}

\begin{lm}\label{eq: b_m (2)}
Let $m\geq2$ be fixed. For every positive integer $n$ such that $m\mid n$, we have that 
\begin{align*}
        b_m(n)=b_m(n-m)+b_m\left(\frac{n}{m}\right).
    \end{align*}
\end{lm}

\begin{proof}
The identity is a consequence of the fact that we can take either $1$ as a part or not. If we do, then there are $b_m(n-m)$ such $m$-ary partitions by Lemma \ref{eq: b_m (1)}. If we do not, then one can divide each part by $m$ and receive an $m$-ary partition of $n/m$.
\end{proof}


Now, we investigate the generalization of the $m$-ary partition function to the rectangular case. In order to do that, let us introduce some notions.

\begin{df}
    Let $b_{i,j}(2,n)$ be the number of partitions of \(2\times n\) rectangle that only use blocks of sizes $1\times i_1$ and $2\times j_1$, where $i_1\in\{1,m,m^2,\ldots,m^i\}$ and $j_1\in\{m,m^2,\ldots,m^j\}$. In particular, $b_{i,0}(2,n)$ enumerates  partitions of \(2\times n\) rectangle that only take into account blocks of sizes $1\times i_1$ for $i_1\in\{1,m,m^2,\ldots,m^i\}$. Further, we set $b(2,n):=b_{\infty,\infty}(2,n)$, when all the blocks of sizes $1\times i_1$ and $2\times j_1$ are allowed for $i_1\in\{m^k:k\in\mathbb{N}_0\}$ and $j_1\in\{m^k:k\in\mathbb{N}\}$.
\end{df}
Once again, let us notice that
\begin{align}\label{eq: generalization of m-ary}
    b_{i,j}(2,n)=b_{i,j}(2,n-m^j)+b_{i,j-1}(2,n)
\end{align}
for $j\geq1$, which is a direct consequence of the fact that one can take $2\times m^j$ as a block of a partition or not. Because of the above, let us restrict ourselves only to the case of $b_{i,0}(2,n)$. We have the following.

\begin{thm}\label{thm: 2-m-ary recurrence}
    Let $i,m\geq2$ be fixed. For every $n\geq1$, we have that
    \begin{align*}
        b_{i,0}(2,n)=b_{i-1,0}(2,n)+\sum_{u=0}^{i-1}b_{u,0}(2,n-m^{i})b_m(m^{i-u}-m)+b_{i,0}(2,n-m^{i}).
    \end{align*}
\end{thm}

\begin{proof}
    Let $i,m\geq2$ and $n\geq1$ be fixed. First, let us observe that we can either take \(1\times m^i\) block as a part of our partition or not. If we do not, then we get exactly $b_{i-1,0}(2,n)$ ways to fill the $2\times n$ rectangle. Otherwise, we can put the $1\times m^i$ block in the south-east corner of the rectangle. We can also assume that the blocks are filled in a non-decreasing order from the left border to the right one. Moreover, it turns out that we always obtain a partition of the $2\times m^i$ rectangle at the end that is because, we restrict ourselves only to the blocks of sizes $1\times m^t$ for some $0\leq t\leq i$. Therefore, one needs to consider a few possibilities to fill that $2\times m^i$ rectangle. Here, we examine which block is located in the north-west corner of the $2\times m^i$ rectangle. There are a few possibilities: $1\times 1,1\times m,\ldots,1\times m^{i}$. Let us suppose that $1\times m^s$ is that pointed block for some $0\leq s\leq i$. There are two main consequences of such a choice.  The first one states that the not covered part of the $2\times n$ rectangle (that is the $2\times n-m^i$ piece) can be filled in $b_{s,0}(2,n-m^j)$ ways, as the blocks of sizes at most $1\times m^s$ can only be used to do that. The second one indicates that the remaining part of the $2\times m^i$ rectangle must be filled by blocks of sizes at least $1\times m^s$. In order to calculate in how many ways one can do that, let us systematically apply Lemma \ref{eq: b_m (2)} to deduce the following chain of equalities:
    \begin{align*}
    b_m(m^i)&=b_m(m^i-m)+b_m(m^{i-1})=b_m(m^i-m)+b_m(m^{i-1}-m)+b_m(m^{i-2})=\\
    \cdots&=b_m(m^i-m)+\cdots+b_m(m^2-m)+b_m(m-m)+b_m(1).
    \end{align*}
    By the proof of Lemma \ref{eq: b_m (2)}, we know that the value $b_m(m^{i-s}-m)$ is the number of $m$-ary partitions of $m^i$, where the smallest part is equal to $m^s$. The term $b_m(1)=1$ corresponds to the case, when the smallest part is $m^i$. In conclusion, for every $0\leq s\leq i-1$, we gather $b_{s,0}(2,n-m^i)b_m(m^{i-s}-m)$ ways to fill the $2\times n$ rectangle in such a way that the $1\times m^s$ block is put in the north-west corner of the last $2\times m^i$ piece of the rectangle. If $s=i$, then there are $b_m(1)b_{i,0}(2,n-m^i)$ possibilities to do that. This ends the proof.
\end{proof}

Theorem \ref{thm: 2-m-ary recurrence} requires that $i\geq2$. The formula for $b_{i,0}(2,n)=1$ is clear. If $i=1$, then the appropriate formula can be found in Remark \ref{remark: 2 parts}.  

At the moment, we plan to apply Theorem \ref{thm: 2-m-ary recurrence} to derive some arithmetic properties of $b_{i,0}(2,n)$ and $b(2,n)$. More precisely, we present the analogue of Alkauskas' theorem for the $m$-ary partition function (see, \cite{Alkauskas}).

\begin{thm}[Alkauskas]\label{thm: Alkauskas}
    Let $m\geq2$ and $n=\sum\limits_{l=0}^kc_lm^l$, where $c_l\in\{0,1,\ldots,m-~1\}$ for every $l\in\{0,1,\ldots,k\}$, be fixed. Then, we have that
    \begin{align*}
        b_m(n)\equiv\prod_{l=1}^k(c_l+1)\pmod{m}.
    \end{align*}
\end{thm}

It is worth pointing out that the above result was also proved by Andrews et al. \cite{Andrews3}, who used the formal power series approach. There is also a proof in a combinatorial manner due to Edgar \cite{Edgar}. On the other hand, probably the shortest proof of that property was exhibited by Żmija \cite{Żmija}. In the case of $b_{i,0}(2,n)$, the characterization presents as follows.

\begin{pr}\label{proposition: congruence fro b_i,0}
    Let $m\geq3$ and $i\geq0$ be fixed. Suppose further that $n=\sum_{l=0}^{k}c_lm^l$ is the representation of $n$ in base $m$ ($c_s=0$ for $s>k$). Then, we have that
    \begin{align*}
        b_{i,0}(2,n)\equiv\prod_{l=1}^{i}(2c_l+1)\pmod{m}.
    \end{align*}
    Further, for $m=2$ and $i\geq1$, we have that
    \begin{align*}
        b_{i,0}(2,n)\equiv c_0+1\pmod{m}.
    \end{align*}
\end{pr}
\begin{proof}
    At first, let us assume that $m\geq3$ and $i\geq2$ (for $i<2$, the statement can be easily verified). We require the correctness of the claim for all values smaller than $i$. Hence, Theorem \ref{thm: 2-m-ary recurrence} and Theorem \ref{thm: Alkauskas} maintain that
    \begin{align*}
        b_{i,0}(2,n)&=b_{i-1,0}(2,n)+\sum_{u=0}^{i-1}b_{u,0}(2,n-m^{i})b_m(m^{i-u}-m)+b_{i,0}(2,n-m^{i})\\
        &\equiv b_{i-1,0}(2,n)+b_{i-1,0}(2,n-m^{i})+b_{i,0}(2,n-m^{i})\\
        \cdots&\equiv b_{i-1,0}(2,n)+2\sum_{u=0}^{\floor{n/m^i}}b_{i-1,0}(2,n-um^{i})\pmod{m},
    \end{align*}
    where the last equality is the consequence of the fact that $b_{i,0}(2,n-\floor{n/m^i}m^i)=b_{i-1,0}(2,n-\floor{n/m^i}m^i)$. 
    
    Now, let us suppose that $n=\sum_{l=0}^{k}c_lm^l$ and put $c=\prod_{l=1}^{i-1}(2c_l+1)$. Further, the induction hypothesis asserts that
    \begin{align*}
        b_{i,0}(2,n)&\equiv c+2\left(c_km^{k-i}+\cdots+c_i\right)c\equiv(2c_i+1)c\equiv\prod_{l=1}^{i}(2c_l+1)\pmod{m},
    \end{align*}
    as required. This ends the proof for $m\geq3$.

    If $m=2$, then one can apply Proposition \ref{proposition: p_k,l} and easily verify the correctness of the statement for $i<2$. If $i\geq2$, we may repeat the reasoning from the case of $m\geq3$, and obtain the required congruence. The details are left to the interested reader.
\end{proof}

Proportion \ref{proposition: congruence fro b_i,0} is the fundamental ingredient in the proof of the main result of this section.

\begin{thm}\label{thm: generalization of Alkauskas}
    Let $m\geq3$ and $i,j\geq0$ be fixed. Suppose further that $n=\sum_{l=0}^{k}c_lm^l$ is the representation of $n$ in base $m$ ($c_s=0$ for $s>k$). Then, we have that
    \begin{align*}
        b_{i,j}(2,n)\equiv\prod_{l=1}^{i}(2c_l+1)\prod_{l=1}^{j}(c_l+1)\pmod{m}.
    \end{align*}
    Further, for $m=2$, $i,j\geq0$ and $i\neq0$, we have that
    \begin{align*}
        b_{i,j}(2,n)\equiv \prod_{l=0}^j(c_l+1)\pmod{m}.
    \end{align*}
\end{thm}

\begin{proof}
    We prove the result for $m\geq3$. In the case of $m=2$, the proof is analogous, and we leave it as an exercise for the reader. By Proposition \ref{proposition: congruence fro b_i,0} the statement is clear if $j=0$. Thus, let us proceed by induction on $j$. We assume the correctness of the claim for all values between $0\leqslant j_0\leqslant j-1$ and check its validity for $j_0=j$. Equality (\ref{eq: generalization of m-ary}) guarantees that
    \begin{align*}
        b_{i,j}(2,n)=b_{i,j-1}(2,n)+b_{i,j}(2,n-m^j)=\cdots=\sum_{u=0}^{\floor{n/m^j}}b_{i,j-1}(2,n-um^j).
    \end{align*}
    Next, let us demand that the representation of $n$ in base $m$ takes the form $n=\sum_{l=0}^{k}c_lm^l$, and set $d=\prod_{l=1}^i(2c_l+1)\prod_{l=1}^{j-1}(c_l+1)$. The induction hypothesis points out that
    \begin{align*}
        b_{i,j}(2,n)\equiv d+\left(c_km^{k-j}+\cdots +c_j\right)d\equiv\prod_{l=1}^{i}(2c_l+1)\prod_{l=1}^j(c_l+1)\pmod{m},
    \end{align*}
    as required. This concludes the induction step, and thereby ends the proof.
\end{proof}

\section*{Acknowledgments}
We are grateful to Pavlo Yatsyna for useful comments and suggestions.
The first author was supported by the National Science Center grant no.\linebreak 2024/53/N/ST1/01538. The second author was supported by {Charles University} programme PRIMUS/24/SCI/010.

\end{document}